\documentclass[11pt]{article}

\usepackage[a4paper,margin=1in]{geometry}
\usepackage{amsmath,amssymb,amsthm,mathtools}
\usepackage{enumitem}
\usepackage[hidelinks]{hyperref}
\usepackage[nameinlink,noabbrev]{cleveref}
\usepackage{microtype}

\newtheorem{definition}{Definition}[section]

\newtheorem{axiom}[definition]{Axiom}
\newtheorem{proposition}[definition]{Proposition}
\newtheorem{theorem}[definition]{Theorem}
\newtheorem{lemma}[definition]{Lemma}
\newtheorem{remark}[definition]{Remark}
\newtheorem{example}[definition]{Example}

\newcommand{\bits}{\{0,1\}}
\newcommand{\strings}{\bits^*}
\newcommand{\Aev}{\mathcal A}
\newcommand{\World}{\mathcal W}

\newcommand{\Obs}{\mathcal O}

\title{\textbf{Well-Defined but Not Predetermined}\\
\large Persistent Mathematical Objects, Embedded Observers, and the Provability of \(P=NP\)}

\author{Rasoul Ramezanian\\ \small University of Fribourg \\ \small rasoul.ramezanian@unifr.ch}
\date{}

\begin{document}
\maketitle

\begin{abstract} 
We develop a foundational distinction between \emph{well-definedness} and
\emph{predetermination}.  A persistent mathematical or computational object may return a
unique answer whenever a query is made and preserve every previously returned answer, while
its answers to as-yet unqueried inputs are not represented by one completed,
history-independent extension.  We call such an object \emph{well-defined but
non-predetermined}.

The framework is developed through persistent, history-indexed partial functions and set
processes.  Every realized history of such a process is extensionally compatible with a
classical predetermined object, although different incompatible histories may require
different classical completions.   

We next formulate an observer-based semantics.  All observers inhabit one shared
computational world, their interactions modify one common global state, and they have access
to one realized history.  Under axioms of shared history, historical irreversibility, persistence, and
historical indistinguishability, embedded observers cannot establish from the realized
history alone whether its compatible extension was predetermined or progressively formed.

A persistently evolutionary polynomial time computable oracle supplies a concrete example and yields a
deterministic--nondeterministic separation under an explicitly history-sensitive
  semantics.  We then derive a conditional result concerning the provability of \(P=NP\). Suppose that an observer cannot distinguish its actual computational world from a \textit{persistently evolutionary world} in which \(P\neq NP\). Suppose further that the observer's entire internal reasoning process—including any purported proof of \(P=NP\)—is reproduced in that evolutionary world. The same proof process would then lead the observer to conclude \(P=NP\) in a world in which \(P=NP\) is false. Consequently, that process cannot constitute a sound proof of \(P=NP\). Under the proposed observer axioms and the existence of such an indistinguishable evolutionary realization, no embedded observer can therefore possess a sound internal proof of \(P=NP\). If human mathematicians are themselves embedded observers in this sense, the same conclusion applies to us.

The conclusion is an observer-relative unprovability theorem, \textbf{not an unconditional proof} of
classical \(P\neq NP\). 
\end{abstract}

\noindent\textbf{Keywords:} non-predetermined functions; persistent computation; embedded
observers; irreversible histories; \(P\) versus \(NP\); provability; simulation,  Brouwer Choice Sequences.

\medskip
%\noindent\textbf{MSC 2020:} 03D10, 03D15, 03B70, 68Q15.

\section{Introduction}

Classical mathematics normally treats a function or set as extensionally complete.  If
\(f:X\to Y\), then every value \(f(x)\) belongs to one already determined completed graph.  If
\(A\subseteq X\), then every membership proposition \(x\in A\) has a fixed truth value,
independently of whether that question is ever asked.

This ordinary conception combines two logically different properties:
\begin{enumerate}[label=(\roman*)]
  \item \emph{well-definedness}: an input has a unique value;
  \item \emph{predetermination}: all values belong to one completed extension fixed
  independently of \textit{interaction} history.
\end{enumerate}

The purpose of this paper is to separate these properties.  A process may determine a unique
answer when an input is first encountered and preserve that answer forever, while the answer
to a previously unencountered input remains sensitive to the order of future interactions.
Such a process is well-defined along every realized history, but no single
history-independent extension need represent all counterfactual histories.

\begin{quote}
 \textit{Well-definedness may arise from the persistence of already determined values rather than from the predetermination of all future values, while an embedded observer is unable to distinguish whether the values are progressively determined and preserved or fixed in advance.
}
\end{quote}

We develop a framework to formulate explicitly the assumptions concerning embedded observers, and to distinguish carefully between the following three claims:
\begin{enumerate}[label=(\arabic*)]
  \item there exist coherent formal models of persistent non-predetermined computation in which deterministic and nondeterministic polynomial-time computation are separated;
  \item embedded observers are unable to distinguish such non-predetermined models from predetermined ones on the basis of their realized history;
  \item under additional assumptions concerning the preservation and soundness of internal proofs, such observers cannot soundly prove
  \(P=NP\).
\end{enumerate}

The third claim is epistemic.  It is not, by itself, the classical mathematical conclusion
\(P\neq NP\).

The framework has affinities with mathematical un-finished objects as Brouwer's choice sequences
\cite{TroelstraVanDalen1988,VanAtten2007}, interactive computation
\cite{Goldin2000,GoldinSmolkaWegner2006,GoldinWegner2008}, Bostrom's simulation arguments
\cite{Bostrom2003}, and   process-based descriptions of evolving natural systems~\cite{WongEtAl2023}.   

 The main theorem formalizes the following argument. Suppose that an observer cannot distinguish its actual world from a persistently evolutionary world in which \(P\neq NP\). Suppose, moreover, that the observer would carry out exactly the same reasoning and obtain exactly the same purported proof of \(P=NP\) in both worlds. Since \(P=NP\) is false in the evolutionary world, this reasoning cannot constitute a sound proof of \(P=NP\). Thus, under the assumptions introduced in this paper, an embedded observer cannot possess a sound proof of \(P=NP\).

 To provide a rigorous formal foundation for the arguments developed in this paper, Appendix~\ref{sec:formal-beth} introduces a semantic framework combining Beth and Kripke models. Beth models represent the admissible future evolutions of a shared computational world, in which observations, computations, and mathematical proof constructions are themselves events that may change the global state. Kripke models, on the other hand, represent the indistinguishability of different computational worlds from the perspective of embedded observers. By combining these two structures, we formalize the relationships between persistent mathematical objects, historical indistinguishability, internal proof construction, and epistemic soundness. This framework provides a precise logical interpretation of the observer-relative arguments developed in the paper.

 \section{Related Work}
\label{sec:related-work}

\subsection{Brouwer's choice sequences}

The idea that a mathematical object need not be given as a completed totality from the outset
has a well-established precedent in intuitionistic mathematics.  In Brouwer's theory of
choice sequences, an infinite sequence is conceived as being progressively generated rather
than as a completed infinite object whose entries are all fixed in advance; see, for example,
\cite{TroelstraVanDalen1988,VanAtten2007}.  At any finite stage, only an initial segment of
the sequence has been determined, while further entries may be determined as the construction
continues.

This viewpoint is closely related to the distinction between well-definedness and
predetermination considered in the present paper.  A persistent non-predetermined function
similarly develops through time: once a value has been determined, it remains fixed, while
values not yet determined need not belong to a completed extension fixed independently of
future interactions.  Thus, in both settings, mathematical definiteness need not be identified
with the prior existence of a completed infinite object.

There are, however, important differences.  Our framework is not intended as a reformulation
of Brouwerian intuitionism or of the theory of choice sequences.  Here the central object is
an interactive computational process whose future extension may depend on the history and
order of interactions.  Moreover, our main concern is observer-relative: an embedded observer
has access only to one realized history and, under the assumptions introduced in this paper,
may be unable to determine whether the observed values were fixed in advance or progressively
determined and subsequently preserved.  Choice sequences therefore provide an important
foundational precedent for mathematical objects that unfold over time, while the present
framework develops this idea in a persistent, computational, and observer-dependent setting.

\subsection{Simulation and embedded reasoning}

The simulation hypothesis is not assumed in this paper.  Bostrom's simulation argument
concerns probabilistic alternatives involving the development of posthuman civilizations and
the possibility that observers like ourselves inhabit ancestor simulations
\cite{Bostrom2003}.  The present framework requires none of these claims.  It draws only on
the more general methodological possibility that observers may have access to the behavior
generated within a world without having access to the external mechanism by which that world
is realized.

This distinction is particularly relevant when the observer is itself part of the system
under investigation.  In the present framework, observing, reasoning, recording information,
and constructing or verifying a mathematical proof are themselves events occurring within the
modeled world.  The observer therefore does not occupy an external standpoint from which
different realizations of the same history can be inspected and compared.

For this reason, three notions must be kept distinct:
\begin{enumerate}[label=(\roman*)]
\item equality of the histories accessible to the observer;
\item preservation of a particular internal reasoning or proof process;
\item equality of mathematical truth across the underlying worlds.
\end{enumerate}
Two worlds may be indistinguishable from the observer's internal perspective without thereby
being identical as external mathematical structures.  This distinction is central to the
observer-relative argument developed below.

\subsection{Scientific motivation from evolving systems}

Wong et al.\ propose a general account of evolving systems based on combinatorial richness,
the generation of configurations, and selection for function \cite{WongEtAl2023}.  Several
aspects of their framework provide useful conceptual motivation for the present work.

First, their distinction between static and dynamic persistence emphasizes that what persists
through time need not be an unchanging configuration; persistence may instead characterize an
ongoing process.  This is conceptually related to our use of persistence, although persistence
has a more specific meaning here: once a value or membership decision has been realized, it
remains fixed throughout every continuation of the realized history.

Second, Wong et al.\ emphasize the temporal asymmetry and context dependence of evolving
processes.  This is compatible with, although weaker than, the history dependence considered
here, where the future extension of a mathematical object may depend on the order in which
interactions occur.

Third, their discussion of novelty generation and open-ended evolution includes the
possibility that evolving systems construct new possibility spaces rather than merely explore
a fixed space of configurations.  This provides useful motivation for distinguishing a
predetermined but incompletely known future from a process in which future states are
progressively formed.

These connections are motivational rather than deductive.  The theory of evolving systems
does not imply persistent non-predetermination in the technical sense used in this paper, nor
does it establish the observer-indistinguishability assumptions introduced below.

\subsection{Relation to standard complexity barriers}

The relationship between the present construction and classical barriers in complexity theory
requires particular care.  Baker, Gill, and Solovay showed that there exist fixed oracles
relative to which \(P=NP\) and other fixed oracles relative to which \(P\neq NP\)
\cite{BakerGillSolovay1975}.  Consequently, any proof technique that relativizes cannot by
itself resolve the classical \(P\) versus \(NP\) problem.  Natural proofs and algebrization
identify additional limitations on broad classes of proof techniques
\cite{RazborovRudich1997,AaronsonWigderson2009}.

The polynomial-time computable oracle considered in the present paper, which separates deterministic from nondeterministic polynomial-time computation under the proposed semantics, differs fundamentally from a classical fixed oracle.  Its answers are persistent once determined, but its extension is formed through
interaction and may depend on the history and order of queries.  The resulting complexity
classes are therefore defined under an explicitly history-sensitive computational semantics.
The evolutionary separation established in Theorem~\ref{thm:evsep} should consequently not
be interpreted as an ordinary relativized separation for a fixed oracle.

Accordingly, the present framework does not claim that the evolutionary separation alone
constitutes a proof of the classical statement \(P\neq NP\).  Rather, the separation provides
the computational component of the subsequent observer-relative argument.

\section{Well-definedness without predetermination}

\subsection{Histories and partial functions}

Let \(X\) be a countable set of possible inputs and \(Y\) a set of possible outputs.

\begin{definition}[History]
A finite history is a sequence
\[
  H=((x_1,y_1),\ldots,(x_t,y_t))\in(X\times Y)^{<\omega}.
\]
It is \emph{consistent} if \(x_i=x_j\) implies \(y_i=y_j\).  If \(H'\) extends \(H\),
we write \(H\preceq H'\).
\end{definition}

Every consistent history determines a partial function \(f_H\).

\begin{definition}[Persistence]
A history-generating process is \emph{persistent} if
\[
  H\preceq H'\quad\Longrightarrow\quad f_H\subseteq f_{H'}.
\]
\end{definition}

Thus every value once realized remains fixed.

\begin{definition}[Persistent non-predetermination]
A persistent process is \emph{non-predetermined} if there exist admissible histories
\(H,H'\) and an input \(x\) such that
\[
  x\in\operatorname{dom}(f_H)\cap\operatorname{dom}(f_{H'})
  \quad\text{and}\quad
  f_H(x)\neq f_{H'}(x).
\]
\end{definition}

The disagreement occurs across incompatible histories, never inside one realized history.

\subsection{A canonical example}

\begin{example}
Let \(W\) initially be empty.  When \(n\in\mathbb N\) is presented, define
\[
g(n)=
\begin{cases}
z, & (n,z)\in W,\\
|W|+1, & \text{otherwise},
\end{cases}
\]
and in the second case update
\[
W\leftarrow W\cup\{(n,|W|+1)\}.
\]
If the arguments arrive in the order
\[
7,9,1,11,
\]
then \(g\) grows as follows:
\[
g(7)=1,\quad g(9)=2,\quad g(1)=3,\quad g(11)=4.
\]
If they arrive in the order
\[
9,1,7,11,
\]
then \(g\) forms as follows:
\[
g(9)=1,\quad g(1)=2,\quad g(7)=3,\quad g(11)=4.
\]
Repeated presentation of an argument always returns the value previously assigned to it.
\end{example}

The process is not random in the relevant sense: given the current state and next input, the
next output is fixed.  What is absent is a single extension independent of the future input
history. The function \(g\) is a mathematical object whose values are progressively determined over time, in a manner analogous to Brouwer's choice sequences, rather than being specified by a completed extension from the outset.

\subsection{Persistent set processes}

Let \(\Sigma=\bits\).  At history \(H\), define
\[
A_H^+=\{x\in\strings:\chi_H(x)=1\},\qquad
A_H^-=\{x\in\strings:\chi_H(x)=0\},
\]
and
\[
U_H=\strings\setminus(A_H^+\cup A_H^-).
\]

\begin{definition}[Persistent set process]
A persistent set process is a family of partial characteristic functions
\[
  \chi_H:\operatorname{dom}(\chi_H)\to\bits
\]
indexed by admissible histories, such that
\[
  H\preceq H'\quad\Longrightarrow\quad \chi_H\subseteq\chi_{H'}.
\]
\end{definition}

\begin{proposition}[Finite trace completion]\label{prop:completion}
For every finite consistent history \(H\), there exists a classical set
\(B_H\subseteq\strings\) whose membership answers agree with all observations in \(H\).
\end{proposition}

\begin{proof}
Choose any \(C\subseteq U_H\) and define
\[
  B_H=A_H^+\cup C.
\]
For every queried string \(x\), membership in \(B_H\) agrees with the answer in \(H\).
\end{proof}

\begin{remark}
The proposition establishes extensional compatibility along one history.  It does not provide
one static set valid across all possible histories.
\end{remark}

\section{Observers in one shared computational world}

Let
\[
  \World=(S,s_0,\tau)
\]
be a computational world.  Observers, computers, records, and proof activities are processes
within \(\World\).  A realized history is
\[
s_0\xrightarrow{e_1}s_1\xrightarrow{e_2}s_2\longrightarrow\cdots .
\]

\begin{axiom}[Shared history]\label{ax:shared}
All observers in one world interact with one common global state.  An interaction by one
observer becomes part of the state encountered by every later observer.
\end{axiom}

\begin{axiom}[Historical irreversibility]\label{ax:irreversible}
An embedded observer cannot return the entire world to a previous total state (s) while retaining information acquired after that state. Restoring the entire world to (s) would necessarily restore the observer, including their memory and all available records, to their state at (s). Consequently, the observer cannot experience two different continuations of exactly the same total past state and retain enough information to compare them.
\end{axiom}

\begin{axiom}[Persistence of realized facts]\label{ax:persist}
Once the world has realized a membership or functional value, every later realized state
preserves it.
\end{axiom}

\begin{definition}[Internal historical equivalence]
Two world-histories are internally historically equivalent for an observer \(O\) if the
observer has the same accessible sequence of observations, actions, memories, written symbols,
proof states, and verification states in both.
\end{definition}

\begin{axiom}[Historical indistinguishability]\label{ax:ind}
If two worlds generate internally indistinguishable histories for an observer (O), then, on the basis of that history alone, (O) cannot determine which of the two worlds it inhabits.
\end{axiom}

%The axiom does not say that the worlds satisfy the same external formulas.  It says that the observer cannot identify the implementation from the shared internal trace.

\subsection{Why many observers do not defeat indistinguishability}

The observer argument does not depend on there being only one observer.  Let
\[
  \Obs=\{O_1,O_2,\ldots\}
\]
be any population of observers.  By Axiom~\ref{ax:shared}, their interventions form one
interleaved global history.  If \(O_1\) changes the state, \(O_2\) encounters the changed
state.

To distinguish predetermination directly from persistent formation, the observers would need
to compare incompatible continuations from the numerically identical total past while
retaining information from both continuations.  This is what Axiom~\ref{ax:irreversible}
excludes.

%Creating independent copies is not equivalent.  Copies are distinct systems, not two simultaneously remembered futures of one numerically identical total state.

\section{A persistently evolutionary polynomial time  oracle}\label{sec:oracleA}
In this section, we introduce a  persistently evolutionary polynomial time computable oracle denoted by \(\Aev\).

Let
\[
M=(Q,\bits,\delta,q_0,F)
\]
be a partial deterministic finite automaton.  Initially
\[
Q=\{q_0\},\qquad \delta=\varnothing,\qquad F=\varnothing.
\]
where \(Q\) shows states of the automaton, \( \delta\)  denotes the partial transition function, and  \(F\) refers to accepting states.

On a query \(x=a_1\cdots a_n\), follow the defined transitions.

 \begin{enumerate}[label=\textbf{Rule \arabic*.}]
    \item If, following transitions in \( \delta\), all of \(x\) is read and the final state lies in \(F\), accept without change.
    \item If, following transitions in \( \delta\), all of \(x\) is read and the final state \(q\notin F\), reject if \(q\) has a
    one-symbol transition to an accepting state; otherwise add \(q\) to \(F\) and accept.
    \item If, following transitions in \( \delta\),  the automaton crashes on an undefined transition, add a new path for the unread
    suffix (adding new states and transitions) to complete reading \(x\), and make the final new state accepting.
\end{enumerate}

Call the resulting process \(\Aev\).

\begin{proposition}[Persistence]\label{prop:persistent-oracle}
Every membership answer returned by \(\Aev\) is permanent.
\end{proposition}

\begin{proof}
An accepted string ends at an accepting state, whose status is never removed.  A rejected
string ends at a nonaccepting state having a one-symbol transition to an accepting state;
neither that transition nor the nonaccepting status is later removed.
\end{proof}

\begin{proposition}[History dependence]\label{prop:history}
The extension generated by \(\Aev\) is not predetermined.
\end{proposition}

\begin{proof}
If \(10\) is queried first, Rule 3 creates an accepting path for \(10\).  If \(101\) is
queried first, the state reached after \(10\) is nonaccepting and has a one-symbol transition
to an accepting state, so a later query \(10\) is rejected.
\end{proof}

\begin{proposition}[Response cost]
Under unit-cost access to transitions and unit cost per state or transition update, every query
\(x\) is processed in \(O(|x|)\) operations.
\end{proposition}

\begin{proof}
At most \(|x|\) transitions are traversed. If a path is added, at most \(|x|\) states and
transitions are created.
\end{proof}

The automaton above is the formal representation of \(\mathcal A\) used
throughout the paper.  It is also possible to realize the same process as a
persistently growing program.  In this representation, every finite binary
prefix corresponds to a directory containing a constant-size
\texttt{node.py} file with three binary values recording whether the
corresponding state is accepting and whether its \(0\)- and \(1\)-branches
have already been generated.  Processing an input follows only the
directories corresponding to its successive prefixes, while Rules~2 and~3
modify or extend this code when necessary.  Appendix~\ref{app:growing-code}
describes this implementation and shows that each query is processed in
time polynomial in the length of the input.

\section{History-sensitive complexity}\label{sec:HSC}

A deterministic machine interacts sequentially with the current global state of \(\Aev\).
Let \(P_{\mathrm{ev}}^{\Aev}\) denote languages decided in polynomial time equipped with the oracle \(\Aev\).

For nondeterminism, we use certificates~\cite{AroraBarak2009}.

\begin{definition}
A language \(L\) belongs to \(NP_{\mathrm{ev}}^{\Aev}\) if there are
\(J\in P_{\mathrm{ev}}^{\Aev}\) and a polynomial \(q\) such that
\[
x\in L
\quad\Longleftrightarrow\quad
\exists y\bigl(|y|\le q(|x|)\land (x,y)\in J\bigr).
\]
\end{definition}

Let
\[
L_{\Aev}=\{x:\Aev\text{ accepts }x\}
\]
 and
\[
L_{\exists}
=
\{w:\exists x\,(|x|=|w|\land x\in L_{\Aev})\}.
\]

\begin{proposition}\label{prop:belongtoNP}
\(L_{\exists}\in NP_{\mathrm{ev}}^{\Aev}\).
\end{proposition}

\begin{proof}
Use \(x\) as the certificate, verify \(|x|=|w|\), and query \(\Aev\).
\end{proof}

\begin{definition}[Continuation-robust decision]
A decision made at history \(H\) is continuation-robust if it remains correct in every
admissible extension \(H'\succeq H\).
\end{definition}

\begin{theorem}[Evolutionary separation]\label{thm:evsep}
Under continuation-robust semantics,
\[
P_{\mathrm{ev}}^{\Aev}\neq NP_{\mathrm{ev}}^{\Aev}.
\]
\end{theorem}

\begin{proof}
[Proof sketch]
By Proposition~\ref{prop:belongtoNP}, \(L_{\exists}\in NP^{\Aev}_{\mathrm{ev}}\).  
 Suppose, toward a contradiction, that a deterministic polynomial-time procedure \(D\) continuation-robustly decides \(L_{\exists}\). Let \(p\) be a polynomial bounding the running time of \(D\). Choose \(n\) sufficiently large so that
 \( p(n)<2^n \) 
 and no string of length at least \(n\) has previously been queried to \(\Aev\) by any embedded observer. 
 
 Fix an arbitrary \[ w\in\{0,1\}^n. \] We run \(D\) on \(w\) in the shared computational world. All observers and computational procedures interact with the same global state of \(\Aev\). Hence, while \(D\) is running, other embedded observers may interleave their interactions with those of \(D\) and submit additional queries to \(\Aev\). These queries become part of the same persistent global history. 
 
 Whenever \(D\) is about to submit an oracle query \(y\) with \(|y|\ge n\), let \(x\) be the prefix of \(y\) of length \(n\). If no previous interaction has already determined the relevant branch through \(x\), another embedded observer first queries \(x0\). By the rules defining \(\Aev\), this makes \(x0\) accepted and causes \(x\) itself to be rejected permanently. The query of \(D\) is then processed in the resulting global state. Since \(D\) terminates within \(p(n)\) steps, it makes at most \(p(n)\) oracle queries. Consequently, the above construction can involve at most \(p(n)\) distinct length-\(n\) prefixes. Because \( p(n)<2^n, \) at least one string \( z\in\{0,1\}^n \) remains untouched: no query occurring during the computation has \(z\) as its length-\(n\) prefix. After \(D\) terminates, there are two cases.
 
 If \(D\) rejects \(w\), choose such an untouched string \(z\). Since no previous query has passed through the branch represented by \(z\), querying \(z\) next causes Rule~3 to create the missing path and make \(z\) accepting. Hence \( z\in L_{\Aev}. \) Since \(|z|=|w|=n\), it follows that \( w\in L_{\exists}. \) Thus there is an admissible continuation of the shared global history in which the negative answer of \(D\) is false, contradicting continuation robustness. 
 
 If \(D\) accepts \(w\), then every length-\(n\) string that arose through the above interleaving has already been made permanently rejected. For each remaining untouched string \( x\in\{0,1\}^n, \) another embedded observer may query \(x0\) before \(x\). By the same rule, this permanently forces \(x\) to be rejected. After these finitely many   interactions, every string of length \(n\) is rejected, so \[ L_{\Aev}\cap\{0,1\}^n=\varnothing. \] Therefore no witness of length \(n\) exists and hence \( w\notin L_{\exists}. \) 
 Thus there is an admissible continuation in which the positive answer of \(D\) is false, again contradicting continuation robustness.
 
 The essential point is that a polynomial-time procedure can interact with only polynomially many parts of an exponentially large space of possible length-\(n\) witnesses. When \(D\) terminates, it therefore cannot have fixed enough of the future evolution of the shared global state to make either answer invariant under all admissible continuations. Hence no deterministic polynomial-time procedure continuation-robustly decides \(L_{\exists}\). 
 
 Since \( L_{\exists}\in NP^{\Aev}_{\mathrm{ev}}, \) we conclude that \[ P^{\Aev}_{\mathrm{ev}} \neq NP^{\Aev}_{\mathrm{ev}}. \]
The complete argument is given in Appendix~\ref{app:evolutionary-separation}.
\end{proof}

The proof rests on two related ideas. First, the global state is shared:
\(D\) does not compute in isolation, and a proposition established by an
embedded observer must remain correct regardless of how other embedded
observers may simultaneously interleave their interactions with the same global
state. Second,   a polynomial-time procedure can
interact with only polynomially many parts of an exponentially large space
of possible length-\(n\) witnesses. When \(D\) terminates, it therefore
cannot have fixed enough of the future evolution of the shared global state
to make either answer invariant under all admissible continuations. In this
sense, \(D\) cannot close the relevant future. Hence no deterministic
polynomial-time procedure continuation-robustly decides \(L_{\exists}\).

\begin{remark}
This theorem concerns the explicitly defined evolutionary semantics.  It is not a standard
relativized separation for a fixed oracle set.  The distinction is essential in view of the
relativization barrier \cite{BakerGillSolovay1975}.
\end{remark}

\section{Proof as an internal event}

A formal derivation available to an embedded observer must be written, stored, read, checked,
or communicated.  These activities are events in the world-history.

\begin{definition}[Internal proof realization]
An internal realization of a finite derivation \(\pi\) consists of the observer's accessible
sequence of symbol states, rule applications, memory states, and verification states through
which \(\pi\) is produced and accepted as a proof.
\end{definition}

\begin{axiom}[Internal proof preservation]\label{ax:proof}
Suppose an observer has internally historically equivalent histories in two worlds, and the
same finite derivation \(\pi\), formal axioms, inference rules, intended semantics, and
verification process are reproduced in both.  If \(\pi\) is sound under that intended
semantics, its conclusion cannot be false in one of the two worlds.
\end{axiom}

Here, saying that \(\pi\) is \emph{sound under the intended semantics} means that its conclusion must hold in every world that is internally indistinguishable to the embedded observer from the world in which the proof is carried out. More formally, let
\[
\World \equiv_O \World'
\]
mean that the worlds \(\World\) and 
\(\World'\) are internally historically indistinguishable for the embedded observer \(O\), with the relevant proof realization and its intended semantics preserved. If \(O\) realizes a sound proof \(\pi\) of a proposition \(\varphi\) in \(\World\), then soundness requires
\[
\World' \models \varphi
\qquad
\text{for every }\World' \text{ such that } \World'\equiv_O\World.
\]
Equivalently,
\[
O\text{ realizes a sound proof of }\varphi\text{ in }\World
\quad\Longrightarrow\quad
\forall \World'~
\bigl(\World'\equiv_O\World
~\Longrightarrow~
\World'\models\varphi\bigr).
\]
Thus, a proof realized by an embedded observer is sound only if its conclusion is invariant across all relevant worlds that are internally indistinguishable to that observer.

%This axiom does not say that historically indistinguishable worlds satisfy every same sentence. It concerns a particular proof process whose complete internally relevant realization is preserved.

\subsection{Observer-compatible evolutionary realizations}

\begin{definition}[Observer-compatible realization]
Let \(O\) be an embedded observer with realized history \(H\). A world
\(\World'\) is an \emph{observer-compatible realization} of \(H\) if
\(O\)'s complete internally accessible history in \(\World'\) is identical
to \(H\). In particular, \(O\) has the same observations, memories,
interactions, records, and internally realized reasoning processes in both
worlds.
\end{definition}

\begin{proposition}[Finite-history realization]
\label{prop:finite-realization}
Let \(H\) be any finite consistent history accessible to an embedded
observer \(O\). Then there exist both a predetermined realization and a
persistent non-predetermined realization that are observer-compatible with
\(H\).
\end{proposition}

\begin{proof}
Since \(H\) is finite and consistent, the values determined along \(H\)
define a finite partial function. A predetermined realization is obtained
by extending this partial function arbitrarily to a total function on the
entire domain.

For the non-predetermined realization, preserve every value already fixed
by \(H\), while leaving values not determined by \(H\) to be assigned by a
persistent evolutionary process. The process agrees with \(H\) on every
interaction already observed by \(O\), but previously undetermined values
may be assigned differently in different admissible continuations.
Therefore the two realizations generate the same internally accessible
history for \(O\), although one is predetermined and the other is not.
\end{proof}

%Condition (iv) is substantive.  It cannot be replaced merely by the observation that \(P_{\mathrm{ev}}^{\Aev}\neq NP_{\mathrm{ev}}^{\Aev}\) if the purported proof concerns different, classical complexity classes.  A fully formal application must align the sentence proved with the sentence falsified. This alignment requirement identifies the central technical burden of the argument.

%\section{The main observer-relative theorem}

A sound internal proof  means
an internally justified conclusion about computation in the
observer's own world. Such a justification must support both
the formal derivation and the claim that its interpretation
faithfully represents computation in that world. Producing a
derivation in a formal theory alone does not meet this requirement.

\begin{theorem}[Embedded-observer unprovability]\label{thm:main}
Assume axioms~\ref{ax:shared},\ref{ax:irreversible},\ref{ax:persist},\ref{ax:ind},\ref{ax:proof}.  Let \(O\) be an embedded
mathematical observer.  Suppose \(O\)'s realized history has an observer-compatible
evolutionary realization in the sense above.  Then \(O\) cannot possess a sound internal proof
of \(P=NP\).
\end{theorem}

\begin{proof}
Assume that \(O\) possesses a sound internal proof \(\pi\) of \(P=NP\). 

By Theorem~\ref{thm:evsep},
there is an observer-compatible evolutionary realization in which \(O\)'s complete internal
history and the proof realization of \(\pi\) are reproduced, with the same formal system,
intended semantics, and verification process.

By Axiom \ref{ax:proof}, a sound internal proof
process reproduced with the same intended semantics cannot establish a sentence false in one
of the two worlds.  Therefore \(\pi\) is not sound, contrary to assumption. For a formal proof see Appendix~\ref{sec:formal-beth}.
\end{proof}

\subsection{What follows, and what does not}

The result is conditional.
The theorem has the logical form
\[
\mathsf{Axioms}
\quad\Longrightarrow\quad
\text{no sound internal proof of }P=NP.
\]
The antecedent is not established by ordinary complexity theory.

 Unprovability does not imply falsity.
In general,
\[
\nvdash P=NP
\]
does not imply
\[
P\neq NP.
\]
A true sentence may be unprovable in a particular system.  No truth--provability completeness
principle is assumed here.

 Evolutionary and classical complexity must not be conflated.
The separation in Theorem \ref{thm:evsep} concerns continuation-robust, history-sensitive classes.
A proof about standard Turing-machine \(P\) and \(NP\) is not automatically refuted by a
separation between different classes.  

 External distinguishability is compatible with internal indistinguishability.
An external metalanguage may distinguish predetermined from non-predetermined worlds by
quantifying over alternative histories.  The observer claim is only that inhabitants confined
to one realized history lack that access.

\section{Conclusion}

A mathematical object, similar to Brouwer's choice
sequences, may be well-defined because every realized value is unique and
persistent, without possessing one completed extension.

We then consider observers who are themselves part of the world they observe. All such observers interact with the same evolving world and therefore share a single realized history. Because they cannot return the entire world to exactly the same past state while retaining knowledge of what happened afterward, they cannot observe and compare two different continuations of that same past. Consequently, from the history available to them, the observers may be unable to determine whether the values they observe were fixed in advance or were progressively determined through interaction.

A persistently evolutionary oracle provides a concrete computational model of this idea. Within the history-sensitive semantics introduced in this paper, the model yields a separation between deterministic and nondeterministic polynomial-time computation. We then consider its consequence for the provability of \(P=NP\). Suppose that an observer cannot distinguish its actual world from an evolutionary world in which \(P\neq NP\), and suppose that the observer would carry out exactly the same reasoning and obtain the same purported proof of \(P=NP\) in both worlds. Since \(P=NP\) is false in the evolutionary world, that reasoning cannot constitute a sound proof of \(P=NP\).

The principal conclusion is therefore conditional and epistemic:
\[
  \boxed{\text{embedded observers cannot soundly prove \(P=NP\).}}
\]

\begin{quote}
    If we are embedded observers in a world satisfying the stated assumptions, this limitation applies to us regardless of whether the actual world is predetermined or non-predetermined. The argument does not require the actual world to be non-predetermined; it requires only that, from our position within the world, we cannot distinguish a predetermined world from an observationally compatible non-predetermined one.

\end{quote}

The argument does not prove classical \(P\neq NP\).  Its proposed contribution lies instead
in the foundations of computation and mathematical knowledge: it isolates predetermination as
an assumption distinct from well-definedness and asks what mathematical proof can establish
when the reasoner is itself a process inside one irreversible computational history.

%\section*{Acknowledgments} The author used generative AI tools for English-language editing and stylistic improvement. The mathematical ideas, definitions, arguments, and results presented in this paper are the author's own.

\appendix
\section{A Growing-Code Implementation of \texorpdfstring{\(\mathcal A\)}{A}}
\label{app:growing-code}

The process \(\mathcal A\) was defined in Section~\ref{sec:oracleA} by means of a
persistently growing finite automaton.  The same process can be represented
more concretely by a program whose code grows as new inputs are processed.
We describe here a simple  implementation.

The purpose of this implementation is not to introduce a different oracle.
Rather, it provides another representation of the same process
\(\mathcal A\).  States of the automaton are represented by directories,
transitions are represented by subdirectories, and the accepting status of
a state is represented by a small file stored in the corresponding
directory.

\subsubsection*{The directory structure}

The generated program is organized as a binary tree of directories.  Each
directory corresponds to one finite binary string, interpreted as a prefix
of an input.

The root directory, denoted by
\texttt{A\_code/}, corresponds to the empty string \(\epsilon\).  Its
subdirectory \texttt{0/}, when present, corresponds to the prefix \(0\),
and its subdirectory \texttt{1/} corresponds to the prefix \(1\).
Recursively, if a directory represents a prefix \(u\), then its
subdirectories
\[
\texttt{0/}
\qquad\text{and}\qquad
\texttt{1/}
\]
represent \(u0\) and \(u1\), respectively.

For example, if the path corresponding to the string \(101\) has already
been generated, the relevant part of the program has the form
\begin{verbatim}
A_code/
    node.py
    1/
        node.py
        0/
            node.py
            1/
                node.py
\end{verbatim}

Thus, the successive directories represent
\[
\epsilon,\qquad
1,\qquad
10,\qquad
101.
\]

\subsubsection*{The contents of \texttt{node.py}}

Every directory contains a small file called \texttt{node.py}.  The file
contains exactly three binary values:
\[
(A,Z,O)\in\{0,1\}^3.
\]

Their meanings are as follows:
\begin{enumerate}[label=(\roman*)]
    \item \(A=1\) if and only if the prefix represented by the current
    directory is accepting;

    \item \(Z=1\) if and only if the current directory has a
    \texttt{0/} subdirectory;

    \item \(O=1\) if and only if the current directory has a
    \texttt{1/} subdirectory.
\end{enumerate}

For example, a file may have the form
\begin{verbatim}
ACCEPTING = 0
ZERO_CHILD = 1
ONE_CHILD = 0
\end{verbatim}

This means that the corresponding prefix is currently nonaccepting, that
its \(0\)-extension has already been generated, and that its
\(1\)-extension has not yet been generated.

Hence each \texttt{node.py} has constant size.  The information accumulated
by \(\mathcal A\) is not stored in one growing table that must be searched.
Instead, it is distributed over the growing binary directory structure.

The correspondence with the automaton representation is therefore:
\[
\begin{array}{c|c}
\text{Automaton representation}
&
\text{Growing-code representation}
\\ \hline
\text{state corresponding to }u
&
\text{directory corresponding to }u
\\
u\xrightarrow{0}u0
&
\texttt{0/} subdirectory
\\
u\xrightarrow{1}u1
&
\texttt{1/} subdirectory
\\
u\in F
&
\texttt{ACCEPTING = 1}
\end{array}
\]

\subsubsection*{Processing an input}

Let
\[
x=a_1a_2\cdots a_n,
\qquad a_i\in\{0,1\}.
\]

The program begins in the root directory.  After reading \(a_1\), it
examines the corresponding bit in the root \texttt{node.py}.  If the
required subdirectory exists, the computation moves into that directory.
It then reads \(a_2\) and repeats the same procedure.

Thus, as long as all required branches already exist, the program visits
only the directories corresponding to
\[
\epsilon,\quad
a_1,\quad
a_1a_2,\quad
\ldots,\quad
a_1a_2\cdots a_n.
\]

In particular, the program never searches through directories belonging to
prefixes unrelated to the current input.

The behavior at the end of this traversal corresponds exactly to
Rules~1--3 of Section~5.

\subsubsection*{Implementation of Rule~1}

Suppose that the complete path corresponding to \(x\) already exists.
After reading the last symbol \(a_n\), the program reaches the directory
representing \(x\).

If its \texttt{node.py} contains
\begin{verbatim}
ACCEPTING = 1
\end{verbatim}
then the program returns \texttt{True}.  No directory or file is modified.

This is exactly Rule~1 of the automaton representation.

Because the accepting bit is never changed from \(1\) back to \(0\), the
answer remains permanent.

\subsubsection*{Implementation of Rule~2}

Suppose that the complete path corresponding to \(x\) exists, but the
terminal file contains
\begin{verbatim}
ACCEPTING = 0
\end{verbatim}

The program then needs to determine whether the state corresponding to
\(x\) has a one-symbol transition to an accepting state.

There are at most two possibilities: \(x0\) and \(x1\).

If
\texttt{ZERO\_CHILD = 1}, the program reads the
\texttt{node.py} contained in the \texttt{0/} subdirectory and checks its
accepting bit.  Similarly, if
\texttt{ONE\_CHILD = 1}, it checks the file in the \texttt{1/}
subdirectory.

If either immediate child has
\begin{verbatim}
ACCEPTING = 1
\end{verbatim}
then \(\mathcal A\) returns \texttt{False}.  The current file does not need
to be modified.

Otherwise, neither one-symbol extension is accepting.  In this case,
Rule~2 requires the current state to become accepting.  The program
therefore changes only one bit in the current \texttt{node.py}:
\begin{verbatim}
ACCEPTING = 0
\end{verbatim}
is replaced by
\begin{verbatim}
ACCEPTING = 1.
\end{verbatim}
The program then returns \texttt{True}.

Thus Rule~2 modifies at most one existing node file.

\subsubsection*{Implementation of Rule~3}

Suppose that while reading
\[
x=a_1\cdots a_n
\]
the program has already read
\[
a_1\cdots a_i
\]
but discovers that the subdirectory corresponding to \(a_{i+1}\) does not
exist.

This is the code representation of an undefined transition in the
automaton, and therefore Rule~3 applies.

The program creates the missing directory corresponding to
\(a_1\cdots a_{i+1}\), together with its \texttt{node.py}.  It then updates
the appropriate child bit in the parent file.

For example, if \(a_{i+1}=0\), the parent file is changed from
\begin{verbatim}
ZERO_CHILD = 0
\end{verbatim}
to
\begin{verbatim}
ZERO_CHILD = 1.
\end{verbatim}

The program repeats this construction for every remaining symbol of
\[
a_{i+1}\cdots a_n.
\]

The last generated directory represents the complete input \(x\).  Its
\texttt{node.py} is created with
\begin{verbatim}
ACCEPTING = 1
\end{verbatim}
and the program returns \texttt{True}.

Thus Rule~3 literally enlarges the generated program by adding the missing
branch of the binary directory tree.

\subsubsection*{Example}

Suppose that initially the program consists only of
\begin{verbatim}
A_code/
    node.py
\end{verbatim}
with
\begin{verbatim}
ACCEPTING = 0
ZERO_CHILD = 0
ONE_CHILD = 0
\end{verbatim}

and the first query is
\[
101.
\]

The required branch does not exist.  Rule~3 therefore creates
\begin{verbatim}
A_code/
    node.py
    1/
        node.py
        0/
            node.py
            1/
                node.py
\end{verbatim}

The final file, corresponding to \(101\), has
\begin{verbatim}
ACCEPTING = 1.
\end{verbatim}

The file corresponding to \(10\), however, has
\begin{verbatim}
ACCEPTING = 0
ZERO_CHILD = 0
ONE_CHILD = 1
\end{verbatim}

because it has a \(1\)-child leading to the accepting node \(101\).

Consequently, if \(10\) is queried next, the program reaches the directory
for \(10\), observes that it is nonaccepting, checks its immediate
\(1\)-child, finds that \(101\) is accepting, and returns
\texttt{False} by Rule~2.

By contrast, if \(10\) had been queried first in the initial state,
Rule~3 would have generated the path for \(10\) and created its final file
with
\begin{verbatim}
ACCEPTING = 1.
\end{verbatim}

Hence the growing-code representation reproduces the same history
dependence as the automaton representation.

\subsubsection*{Polynomial running time}

We finally discuss the running time of this implementation.

Let
\[
n=|x|.
\]

When processing \(x\), the program never scans the entire directory tree.
At every level it reads only the \texttt{node.py} corresponding to the
current prefix and then follows at most one of two possible subdirectories.
Therefore at most
\[
n+1
\]
nodes are visited while reading the input.

Each \texttt{node.py} contains only three binary values and therefore has
constant size.

If Rule~1 applies, the program performs only the traversal of the path of
\(x\), requiring \(O(n)\) node operations.

If Rule~2 applies, after this traversal the program examines at most two
additional node files, namely those corresponding to \(x0\) and \(x1\).
It may then modify one binary value in the file corresponding to \(x\).
Thus Rule~2 adds only a constant number of node operations.

If Rule~3 applies after \(i\) symbols have been read, at most
\[
n-i\le n
\]
new directories and constant-size \texttt{node.py} files are created.
Hence at most \(O(n)\) additional node operations are required.

Consequently, under the abstract cost model in which reading, creating, or
updating one constant-size node file is a unit-cost operation, the total
running time is
\[
T_{\mathcal A}(x)=O(n).
\]

For a literal filesystem implementation, accessing a directory at depth
\(i\)  requires processing a pathname of length \(O(i)\). Therefore, using the
conservative bound \(O(n)\) for each such access and noting that at most
\(O(n)\) nodes are accessed or created, we obtain
\(
T_{\mathcal A}(x)
\) is for sure polynomial time. 
Thus, under either cost model, the implementation runs in polynomial time
in the length of the current input.

Importantly, this bound is independent of the total number of nodes
generated during previous interactions.  The complete program may grow
without bound over time, but processing an input \(x\) requires access only
to the branch determined by the prefixes of \(x\).  Consequently, code
generated on unrelated branches does not increase the asymptotic running
time of the current query.

This growing-code representation therefore provides a concrete
polynomial-time implementation of the persistently evolutionary process
\(\mathcal A\).

\section{Proof of Theorem~\ref{thm:evsep}}
\label{app:evolutionary-separation}

In this appendix, we give the complete proof of
Theorem~\ref{thm:evsep}. The proof uses the fact that all computations
interact with the same persistent oracle state and that no fixed order is
imposed on admissible interactions with the oracle. Thus, while simulating
a candidate deterministic procedure, we may perform additional oracle
queries before allowing one of its queries to be processed. All such
queries become part of the same global history, and every answer already
returned by \(\mathcal A\) remains permanent.

Recall that
\[
L_{\mathcal A}
=
\{x\in\{0,1\}^* :
\mathcal A \text{ accepts }x\},
\]
and
\[
L_{\exists}
=
\left\{
w\in\{0,1\}^* :
\exists x\in\{0,1\}^*
\bigl(
|x|=|w|
\land
x\in L_{\mathcal A}
\bigr)
\right\}.
\]

We first establish two consequences of the evolutionary rules.

\begin{lemma}[Protecting a string]
\label{lem:protect-string}
Let \(H\) be the current history, and let
\(x\in\{0,1\}^n\). If no previously processed query has \(x\) as a prefix and the path corresponding to \(x0\) has not already been determined, then querying \(x0\) before \(x\) makes \(x0\) accepted and x rejected. 
Moreover, both answers remain unchanged in every continuation of \(H\).
\end{lemma}

\begin{proof}
The automaton does not yet contain a complete path labelled
\(x0\). Therefore, when \(x0\) is queried, Rule~3 adds the missing
part of the path and makes its final state accepting. Hence
\[
x0\in L_{\mathcal A}.
\]

Now consider a subsequent query of \(x\). The path labelled \(x\)
ends at a nonaccepting state having a transition labelled \(0\) to
the accepting state reached by \(x0\). Rule~2 therefore causes
\(\mathcal A\) to reject \(x\), so
\[
x\notin L_{\mathcal A}.
\]

By the persistence property of \(\mathcal A\), neither answer can
change in any later history.
\end{proof}

\begin{lemma}[Acceptance of an untouched string]
\label{lem:untouched-string}
Let \(H\) be a history, and let \(z\in\{0,1\}^n\). Suppose that no
query made in \(H\) has \(z\) as its length-\(n\) prefix. Then,
if \(z\) is queried next, \(\mathcal A\) accepts \(z\).
\end{lemma}

\begin{proof}
Since no previous query has \(z\) as its length-\(n\) prefix, no
previous interaction has created a complete path labelled \(z\) or
a path passing through \(z\) to one of its extensions. Consequently,
the automaton crashes at some undefined transition while reading
\(z\). Rule~3 then adds a path for the unread suffix and makes its
final state accepting. Thus,
\[
z\in L_{\mathcal A}.
\]
\end{proof}

\begin{theorem}[Evolutionary separation]
\label{thm:evsep-appendix}
Under continuation-robust semantics,
\[
P_{\mathrm{ev}}^{\mathcal A}
\neq
NP_{\mathrm{ev}}^{\mathcal A}.
\]
\end{theorem}

\begin{proof}
By Proposition~\ref{prop:belongtoNP},
\[
L_{\exists}\in NP_{\mathrm{ev}}^{\mathcal A}.
\]
Suppose, toward a contradiction, that
\[
L_{\exists}\in P_{\mathrm{ev}}^{\mathcal A}.
\]
Then there exists a deterministic oracle procedure \(D\) and a
polynomial \(p\) such that, for every input \(w\),

\begin{enumerate}[label=(\roman*)]
    \item \(D\) terminates within \(p(|w|)\) steps;
    \item \(D\)'s output is correct;
    \item \(D\)'s output is continuation-robust.
\end{enumerate}

Notethat all observers and computational procedures interact with the same global oracle state. Consequently, the computation of \(D\) need not occur in isolation: while \(D\) is running, another embedded observer may submit additional queries to \(\mathcal A\). These interactions modify the same persistent state subsequently encountered by \(D\). Hence the sequence consisting of \(D\)'s oracle queries together with the observer's additional queries forms one admissible global history.

\vspace{0.1cm}

Let \(H_0\) be the current global history, and let
\[
m
=
\max
\left\{
|u| :
u \text{ has previously been queried in }H_0
\right\},
\]
where \(m=0\) if no query has yet been made.

Choose \(n>m\) sufficiently large that
\[
p(n)<2^n.
\]
Fix an arbitrary string
\[
w\in\{0,1\}^n.
\]

We run \(D\) on \(w\) as one process within the   shared oracle \(\mathcal A\) and interleave its oracle interactions with queries submitted by another observer.

\begin{quote}
\textit{Whenever \(D\) is about to submit an oracle query \(y\) with
\(|y|\ge n,\)
let
\(x=\operatorname{pref}_n(y)\)
be the prefix of \(y\) of length \(n\). If \(x\) has not already been
protected, we, acting as another observer in the shared computational world,
first query \(x0\). The query \(y\) submitted by \(D\) is then processed
against the resulting global oracle state. We call such a string \(x\)
\emph{protected}.}
\end{quote}
By Lemma~\ref{lem:protect-string}, every protected string
\(x\in\{0,1\}^n\) is rejected by \(\mathcal A\), and this rejection is
permanent.

Since \(D\) executes for at most \(p(n)\) steps, it makes at most
\(p(n)\) oracle queries. Therefore, at most \(p(n)\) distinct strings
of length \(n\) become protected during the computation. Since
\(
p(n)<2^n,
\)
there exists at least one string
\[
z\in\{0,1\}^n
\]
that is not protected. Equivalently, no oracle query made during the
simulation of \(D\) has \(z\) as its length-\(n\) prefix.

After \(D\) terminates, there are two cases.

\medskip

\noindent
\textbf{Case 1: \(D\) rejects \(w\).}

Choose an unprotected string
\[
z\in\{0,1\}^n.
\]
Because no query made during the computation of \(D\) has \(z\) as
its length-\(n\) prefix, Lemma~\ref{lem:untouched-string} applies.
Querying \(z\) causes \(\mathcal A\) to accept it. Hence
\[
z\in L_{\mathcal A}.
\]

Since \(|z|=|w|=n\), it follows that
\[
w\in L_{\exists}.
\]

Thus, there is an admissible continuation of the history in which
the negative output of \(D\) is false. This contradicts the
continuation-robustness of \(D\)'s decision.

\medskip

\noindent
\textbf{Case 2: \(D\) accepts \(w\).}

Every protected string of length \(n\) has already been forced to be
rejected.

For each remaining unprotected string
\[
x\in\{0,1\}^n,
\]
query \(x0\). Since no previous oracle query has \(x\) as its
length-\(n\) prefix, Lemma~\ref{lem:protect-string} applies. Therefore
each such \(x\) is rejected.

After these finitely many additional interactions, every string of
length \(n\) is rejected:
\[
L_{\mathcal A}\cap\{0,1\}^n
=
\varnothing.
\]

Consequently, there exists no \(x\) satisfying
\[
|x|=|w|
\qquad\text{and}\qquad
x\in L_{\mathcal A}.
\]
Thus,
\[
w\notin L_{\exists}.
\]

Therefore, there is an admissible continuation of the history in which the positive output of \(D\) is false. This again contradicts continuation robustness.
Since \(D\) is assumed to decide \(L_{\exists}\) correctly in the shared computational environment, its correctness must be preserved under every admissible interleaving of interactions with the global oracle state. To refute the existence of such a decider, it is therefore sufficient to construct one admissible interleaving in which the output of \(D\) is not continuation-robust.

\medskip

In both cases, the output of \(D\) can be falsified by an admissible
continuation of the shared oracle history. Hence no deterministic
polynomial-time procedure continuation-robustly decides
\(L_{\exists}\). Therefore,
\[
L_{\exists}\notin P_{\mathrm{ev}}^{\mathcal A}.
\]

Since
\[
L_{\exists}\in NP_{\mathrm{ev}}^{\mathcal A},
\]
we conclude that
\[
P_{\mathrm{ev}}^{\mathcal A}
\neq
NP_{\mathrm{ev}}^{\mathcal A}.
\]
\end{proof}

% ------------------------

\section{Formal Proof: Beth--Kripke Semantics}
\label{sec:formal-beth}

In this section, we introduce Beth--Kripke semantics to formalize two distinct forms of possibility. Within a single computational world, embedded observers interact with and evolve a shared global state, giving rise to different admissible continuations of the same history. We represent these possible future evolutions using Beth models~\cite{Beth1956, Bdalen}. At a second level, distinct computational worlds may be internally indistinguishable to an embedded observer, even though their underlying structures or mechanisms differ. We formalize this observer-relative indistinguishability using Kripke-style epistemic models~\cite{Kripke1963,Hintikka1962}. Beth--Kripke semantics combines these two levels: each possible computational world is represented by a pointed Beth model, while epistemic accessibility relations describe which worlds an embedded observer cannot distinguish. We now develop this framework formally and examine its consequences for the observer axioms and history-sensitive computation.

\subsection{Computational histories as Beth nodes}

\begin{definition}[Shared computational transition system]
\label{def:shared-transition}

A \emph{shared computational transition system} is a tuple
\[
\mathcal W =
(\mathcal S,s_0,\mathcal O,\mathcal E,
\longrightarrow),
\]
where:
\begin{itemize}
    \item $\mathcal S$ is the set of global states
    of the computational world;

    \item $s_0\in\mathcal S$ is the initial global state;

    \item $\mathcal O$ is the set of embedded observers;

    \item $\mathcal E$ is the set of labelled interaction
    events, including computations, observations,
    proof steps, and other actions performed by
    embedded observers;

    \item
    \(
    \longrightarrow\ \subseteq
    \mathcal S\times\mathcal E\times\mathcal S
    \) 
    is the transition relation describing how
    interaction events may evolve the global state;

\end{itemize}

A \emph{finite history} of $\mathcal W$ is a sequence
\[
H=
\left(
s_0\xrightarrow{e_1}s_1
\xrightarrow{e_2}s_2
\cdots
\xrightarrow{e_k}s_k
\right),
\]
where
\[
(s_{i-1},e_i,s_i)\in\longrightarrow
\qquad (1\leq i\leq k).
\]

We write $H\preceq H'$ if $H'$ is an admissible
continuation of $H$.

All embedded observers interact with the same
global state. Consequently, their actions jointly
generate a single realized history of the computational
world, rather than separate histories evolving
independently for each observer.
\end{definition}

Fix an admissible history $H_0$. We consider all possible continuations of
that very history, not a collection of independently restarted experiments.

\begin{definition}[The Beth tree of a computational world]
\label{def:computational-beth}
Let $Q_{\mathcal W,H_0}$ consist of the finite admissible histories extending
$H_0$, ordered by prefix extension $\preceq$. The distinguished root is
$H_0$. For a set $\mathsf{At}$ of propositions concerning persistent atomic facts,
let
\[
 F:Q_{\mathcal W,H_0}\longrightarrow 2^{\mathsf{At}}
\]
record the atomic propositions that have been positively established in
the global history. We require
\[
 H\preceq H'\quad\Longrightarrow\quad F(H)\subseteq F(H').
\tag{F}
\]
The resulting triple
\[
 \Theta_{\mathcal W,H_0}
 =(Q_{\mathcal W,H_0},\preceq,F)
\]
is the \emph{Beth representation} of the world from history $H_0$.
\end{definition}

The points of $Q_{\mathcal W,H_0}$ are histories. A \emph{maximal path} is a maximal linearly ordered subset of the
Beth tree. A \emph{bar} $B$ above $H$ is a set of nodes extending $H$ that
intersects every maximal path through $H$.

Condition (F) is the semantic counterpart of Axiom~\ref{ax:persist}: once  a membership or functional
value has been realized  or other atomic fact has been determined, later interactions cannot
reverse it. Axiom~\ref{ax:shared} is encoded by the single global transition system.
Axiom~\ref{ax:irreversible} has a distinct role: the model treats a realized history as one
increasing chain and does not permit an embedded observer to traverse two
incompatible continuations of an identical total past while retaining
records from both. The tree representation enforces this as a property of
a history.

\subsection{Beth satisfaction and the bar property}

For propositional formulas constructed from $\mathsf{At}$, define forcing
$H\Vdash\varphi$ in $\Theta_{\mathcal W,H_0}$ by the usual Beth clauses:
\begin{align*}
H\Vdash p
 &\iff \text{some bar $B$ above $H$ satisfies $p\in F(J)$ for every $J\in B$},\\
H\Vdash \varphi\land\psi
 &\iff H\Vdash\varphi\text{ and }H\Vdash\psi,\\
H\Vdash \varphi\lor\psi
 &\iff \text{some bar $B$ above $H$ has }J\Vdash\varphi
                  \text{ or }J\Vdash\psi\text{ for all }J\in B,\\
H\Vdash \varphi\rightarrow\psi
 &\iff \forall J\succeq H\;(J\Vdash\varphi\Rightarrow J\Vdash\psi),\\
H\Vdash\neg\varphi
 &\iff \forall J\succeq H\;(J\not\Vdash\varphi).
\end{align*}
In particular, failure of forcing $H\not\Vdash\varphi$ does not in
general imply $H\Vdash\neg\varphi$.

\begin{proposition}[Beth persistence and bar characterization]
\label{prop:beth-basic}
For every Beth formula $\varphi$ and node $H$:
\begin{enumerate}
\item $H\Vdash\varphi$ if and only if there is a bar $B$ above $H$
      such that $J\Vdash\varphi$ for every $J\in B$;
\item $H\not\Vdash\varphi$ if and only if there is a maximal path
      $P$ through $H$ on which no node forces $\varphi$;
\item if $H\preceq J$ and $H\Vdash\varphi$, then $J\Vdash\varphi$.
\end{enumerate}
\end{proposition}

These are the familiar bar and persistence facts of Beth semantics. They apply here because
$\Theta_{\mathcal W,H_0}$ is a Beth model. A bar is a \emph{uniform
future guarantee}: every maximal continuation must eventually pass
through one of its nodes.  

\begin{lemma}[Opposite extensions prevent forcing]
\label{lem:opposite-extensions}
Let $H$ be a node and $\varphi$ any formula. If there are extensions
$H_+,H_-\succeq H$ with
\[
 H_+\Vdash\varphi,\qquad H_-\Vdash\neg\varphi,
\]
then
\[
 H\not\Vdash\varphi
 \quad\text{and}\quad H\not\Vdash\neg\varphi.
\]
Consequently, there is no bar above $H$ all of whose nodes force
$\varphi$, nor a bar all of whose nodes force $\neg\varphi$.
\end{lemma}
\begin{proof}
If $H\Vdash\varphi$, persistence would imply
$H_-\Vdash\varphi$, contradicting $H_-\Vdash\neg\varphi$.
If $H\Vdash\neg\varphi$, persistence would imply
$H_+\Vdash\neg\varphi$, contradicting $H_+\Vdash\varphi$.
The claims about bars follow from Proposition~\ref{prop:beth-basic}(1).
\end{proof}

\subsection{The evolutionary oracle as a Beth model}

We now instantiate the construction for the persistently evolutionary oracle
$\mathcal A$ defined in Section~\ref{sec:oracleA}. Each oracle state is the finite partial
automaton obtained after a sequence of globally interleaved requests. Its
three response rules are those of Section~\ref{sec:oracleA}. For every string $x$, we retain the notation
$\mathcal A(S,x)=(S',b)$, where $b\in\{0,1\}$ is the response and
$S'$ is the updated automaton state. The response and state update are
performed atomically; other observers may submit requests between such
atomic operations.

For every integer $n\geq 1$, introduce an atomic proposition $p_n$ with
\[
 p_n\in F(H)
 \quad\Longleftrightarrow\quad
 \text{some string $x\in\{0,1\}^n$ has received the answer $1$
       by history $H$.}
\tag{1}
\]
Persistence of accepted strings makes this valuation monotone. Thus
$H\Vdash p_n$ means that the existence of an accepted length-$n$ string
is unavoidable in the Beth sense, whereas $H\Vdash\neg p_n$ means that
none can ever be accepted along any continuation.

\begin{lemma}[A permanently rejected length-$n$ string]
\label{lem:beth-protect}
Suppose no previously processed query has $x\in\{0,1\}^n$ as a prefix.
If another observer submits $x0$ before the first query involving the
$x$-branch, then $x0$ is accepted and the node corresponding to $x$
becomes nonaccepting with an accepting immediate child. Whenever $x$ is
queried, it is rejected, and that rejection is permanent.
\end{lemma}
\begin{proof}
The path for $x0$ is not complete, so Rule~3 generates the missing
transitions and makes its terminal state accepting. The state at $x$ is
nonaccepting and has an accepting immediate successor. A query of $x$
therefore meets Rule~2's rejection condition. Neither the accepting child
nor the nonaccepting status of $x$ is subsequently reversed.
\end{proof}

\begin{lemma}[An untouched length-$n$ string]
\label{lem:beth-untouched}
If no prior oracle query has $z\in\{0,1\}^n$ as its length-$n$ prefix,
then a query of $z$ is accepted.
\end{lemma}
\begin{proof}
The complete path for $z$ is absent, so Rule~3 creates the missing
suffix and marks its terminal state accepting.
\end{proof}

\begin{theorem}[Polynomial-time obstruction to a Beth bar]
\label{thm:beth-bar-obstruction}
Let $D$ be an arbitrary deterministic oracle procedure whose running
time on inputs of length $n$ is bounded by a polynomial $q(n)$, with
oracle requests and replies counted as computation steps. Let $H_0$ be
any finite shared history. Then for all sufficiently large $n$, there
is an admissible interleaving of $D$ with another observer, on any
input $w\in\{0,1\}^n$, such that $D$ terminates at a history $H_D$
and
\[
 H_D\not\Vdash p_n
 \qquad\text{and}\qquad
 H_D\not\Vdash\neg p_n.
\tag{2}
\]
In particular, the global history at which $D$ terminates admits no
Beth bar uniformly witnessing either $p_n$ or $\neg p_n$.
\end{theorem}
\begin{proof}
Let $m$ be the greatest length of an oracle query occurring in $H_0$,
with $m=0$ if there has been no query. Choose $n>m$ such that
$q(n)<2^n$. Run $D$ on $w$ in the one shared oracle world.

Whenever $D$ is about to query a word $y$ with $|y|\geq n$, write
$x=\operatorname{pref}_n(y)$. If $x$ has not previously been treated in
this way, first let another observer query $x0$; then allow $D$'s query
$y$ to be processed. Call these length-$n$ prefixes \emph{protected}.
Lemma~\ref{lem:beth-protect} ensures that a protected $x$ can never
be accepted. Queries by $D$ of length less than $n$ cannot directly
accept a length-$n$ string. Moreover, a query with a protected
length-$n$ prefix does not make that prefix accepting. Thus no
length-$n$ string has been accepted at the time $D$ terminates.

There are at most $q(n)$ protected prefixes because $D$ makes at most
$q(n)$ oracle queries. Choose an unprotected string
$z\in\{0,1\}^n$. By construction no query up to this point has $z$
as its length-$n$ prefix.

First construct a positive continuation $H_+\succeq H_D$ by querying
$z$. By Lemma~\ref{lem:beth-untouched}, this request accepts $z$,
so $p_n\in F(H_+)$ and consequently $H_+\Vdash p_n$.

Second construct a negative continuation $H_-\succeq H_D$ by querying
$x0$ and then $x$ for every still unprotected
$x\in\{0,1\}^n$, and by explicitly querying every protected $x$ not
previously queried. All such length-$n$ strings are rejected by
Lemma~\ref{lem:beth-protect}. There are only finitely many of them.
At $H_-$ every length-$n$ string has been rejected, and each answer
persists. Hence no continuation of $H_-$ can ever accept such a string.
Therefore $H_-\Vdash\neg p_n$.

The two extensions $H_+$ and $H_-$ satisfy the hypotheses of
Lemma~\ref{lem:opposite-extensions}. It follows that $H_D$ forces
neither $p_n$ nor $\neg p_n$, and neither formula has a uniform
Beth bar above $H_D$.
\end{proof}
Therefore, no deterministic polynomial-time procedure always forces a stable
answer. For every such procedure, arbitrarily large fresh lengths
admit an interleaving whose terminal history forces neither
$p_n$ nor $\neg p_n$.
Theorem~\ref{thm:beth-bar-obstruction} supplies an admissible
interleaving at whose terminal node neither proposition is forced.

\begin{remark}[Present-state truth versus future-stable truth]
\label{rem:present-vs-stable}

It is important to distinguish statements about the current
history from statements about all possible future continuations.
For example, an observer may correctly establish at history
$H_D$ that no string of length $n$ has been accepted so far.
This is a statement about the already realized history, and
it may be both knowable and provable by the observer.
However, this does not imply
\(
H_D \Vdash \neg p_n.
\)
Under Beth semantics, $\neg p_n$ requires that no admissible
future continuation of $H_D$ can establish the existence of
an accepted string of length $n$.

The obstruction theorem concerns this stronger, future-stable
notion of truth. It shows that a deterministic polynomial-time
procedure cannot guarantee a decision that remains correct
under every admissible continuation and interleaving of
interactions with the shared global state. It does not
prevent an embedded observer from making correct and
provable statements about the finite history already realized.
\end{remark}

\subsection{Beth--Kripke models of observer-compatible worlds}

A Beth tree represents future evolution \emph{within} a world. We now
represent the observer's uncertainty between entire computational worlds.
This is the second, epistemic, level of the construction.

\begin{definition}[Computational Beth--Kripke model]
\label{def:computational-beth-kripke}
A \emph{computational Beth--Kripke model} is
\[
 \mathfrak K=\bigl(S,(\Theta_s)_{s\in S},
                    (\sim_O)_{O\in\mathcal O}\bigr),
\]
where each $s\in S$ is a pointed Beth model
$\Theta_s=(Q_s,\preceq_s,F_s,\alpha_s)$ representing the admissible
continuations of a chosen realized history $H_s$ in a computational
world $\mathcal W_s$. The relation $s\sim_O t$ means that the
histories $H_s$ and $H_t$ are indistinguishable through the internally
accessible observations, memories, records, proof states, and
verification states of $O$. When discussing the correctness of a
particular proof, the class of accessible worlds is further restricted
to those interpreting its axioms, inference rules, and conclusion in
the same specified way. The choice of such an admissible class is
explicit mathematical data, not a consequence of trace equality alone.

\end{definition}
\noindent Throughout this section, \(O\) denotes all embedded observers collectively, since we assume that they share the same internally accessible information, epistemic perspective, and indistinguishability relation. Consequently, \(\sim_O\) represents the common indistinguishability relation of all embedded observers within a computational world.

We have 
\begin{center}
\begin{tabular}{ll}
$s\in S$ & a possible \emph{entire computational world},\\
$H\in Q_s$ & a state or history \emph{inside} one computational world,\\

$H\preceq_s H'$ & a possible later global history of that world,\\
$s\sim_O t$ & two worlds that $O$ cannot distinguish at the chosen snapshot.
\end{tabular}
\end{center}

For formulas containing a knowledge operator, retain the Beth connectives
and define the epistemic forcing clause, as in constructive Beth--Kripke
semantics, by
\[
 H\Vdash_{\mathfrak K} K_O\varphi
\quad\Longleftrightarrow\quad
\forall t\in S\,
\left(s\sim_O t\Longrightarrow
 \forall J\in Q_t\;J\Vdash_{\mathfrak K}\varphi\right),
\qquad(H\in Q_s).
\tag{3}
\]
Because $\alpha_t$ is the root of $Q_t$, persistence makes the condition
on the right equivalent to requiring $\alpha_t\Vdash\varphi$ for each
accessible $t$.

\begin{proposition}[Beth--Kripke bar and persistence facts]
\label{prop:bk-basic}
With the static world-to-world relation $\sim_O$ in
Definition~\ref{def:computational-beth-kripke}, the usual Beth--Kripke
induction yields, for every formula $\varphi$ of the propositional
modal language:
\begin{enumerate}
\item $H\Vdash\varphi$ if and only if a bar above $H$ consists of
      nodes forcing $\varphi$;
\item $H\not\Vdash\varphi$ if and only if some maximal path
      through $H$ has no node forcing $\varphi$;
\item $H\preceq_s J$ and $H\Vdash\varphi$ imply $J\Vdash\varphi$;
\item if $H\Vdash K_O\varphi$ for one $H\in Q_s$, then
      $J\Vdash K_O\varphi$ for every $J\in Q_s$.
\end{enumerate}
\end{proposition}
\begin{proof}
The Beth clauses give the first three statements by the usual
induction on formulas. For the fourth, note that the right-hand side
of (3) depends only on $s$ and its $\sim_O$-accessible Beth worlds,
not on the particular node $H\in Q_s$. These are the corresponding properties of constructive Beth--Kripke forcing.
\end{proof}

% \begin{remark}[A fixed epistemic snapshot]
% \label{rem:bk-snapshot}
% The last clause of Proposition~\ref{prop:bk-basic} has a limitation:
% knowledge is constant across all nodes of a given Beth world under
% this version of the semantics. It therefore models the observer's
% compatible alternatives at a \emph{fixed epistemic snapshot}, not the
% acquisition of new knowledge during later interactions. To model such
% learning one must use a history-indexed accessibility relation
% between pairs $(s,H)$ and $(t,J)$, or explicit epistemic updates, and
% reprove the corresponding forcing lemmas. We employ the fixed-snapshot
% semantics only at the history at which a purported proof or computation
% has been completed.
% \end{remark}

\noindent Historical
indistinguishability (Axiom~\ref{ax:ind}) is represented by $\sim_O$ and its
associated knowledge clause (3).

\subsection{Embedded derivations and soundness at compatible worlds}

In the present framework, mathematical reasoning is itself part of the
evolution of a computational world. An embedded observer does not
construct a proof from outside the world being studied. Rather, writing
symbols, applying inference rules, performing computations, and
verifying a derivation are interaction events that may change the shared
global state.

Consequently, the construction of a proof may move the computational
world from one Beth node to another. It does not, however, change the
Kripke world inhabited by the observer.

More precisely, suppose that an observer $O$ constructs a finite
derivation $\pi$ in a computational world $s\in S$. Let
$\alpha_s$ be the Beth node representing the global history before
the proof construction, and let $\alpha_s^\pi$ represent the history
after the proof has been completed and verified. Then
\[
\alpha_s\preceq_s\alpha_s^\pi.
\]
The proof construction may involve several interaction events,
possibly interleaved with actions of other embedded observers.
All these interactions belong to the same global history.

Throughout this process, the observer remains in the same
computational world $s$. Thus, the Kripke world does not change,
although the current node of its associated Beth model may advance.

\begin{definition}[Faithful internal proof event]
\label{def:faithful-proof}

Let $O$ be an embedded observer in a computational world $s$,
and let
\[
\pi=(\varphi_1,\ldots,\varphi_k)
\]
be a finite derivation from a theory $T$, with
$\varphi_k=\varphi$.

A \emph{faithful internal proof event}, denoted by
\[
\operatorname{Prf}_O(\pi,T,\varphi),
\]
is a finite sequence of interaction events through which
$O$ constructs and verifies $\pi$ according to the axioms
and inference rules of $T$.

The proof event is represented by an extension
\[
\alpha_s\preceq_s\alpha_s^\pi,
\]
where $\alpha_s^\pi$ is the Beth node reached after the
derivation has been completed and verified.

The proof event includes the observer's internally accessible
mathematical reasoning, written symbols, computations,
memory states, and verification process.

\end{definition}

We now consider the relationship between such proof events
in different computational worlds.
Suppose that $s$ and $t$ are internally indistinguishable
to the embedded observers at the relevant proof-complete
histories:
\[
s\sim_O t.
\]

Since internal historical indistinguishability includes the
observers' accessible reasoning and verification histories,
the same completed derivation and verification process must
be reproduced in $t$. In particular, there is a corresponding
Beth node $\alpha_t^\pi$ at which the observer has completed
the same proof $\pi$.

The observer therefore experiences the same proof event in
both worlds, although the underlying computational worlds
may differ.

This establishes preservation of the internal proof event.
To establish preservation of its mathematical conclusion,
we must also specify the intended meaning of soundness.

In the present framework, a proof system is sound in a
computational world only if its axioms and assumptions
hold under the intended semantics of that world and its
inference rules preserve truth. Thus, soundness concerns
not merely the syntactic correctness of a derivation,
but also the semantic validity of the theory from which
the conclusion is derived.

When evaluating the soundness of an embedded proof across
indistinguishable worlds, we additionally require that
these semantic conditions be preserved. In particular,
the axioms and assumptions of the theory must remain
valid in every world relevant to the soundness claim.

\begin{definition}[Uniformly admissible proof context]
\label{def:uniform-proof-context}

Suppose that an observer $O$ completes a faithful
internal proof event
\[
\operatorname{Prf}_O(\pi,T,\varphi)
\]
in a computational world $s$, reaching the Beth node
$\alpha_s^\pi$.

Let
\[
\mathcal C_O(s,\pi)
=
\{t\in S:s\sim_O t\}
\]
be the class of computational worlds internally
indistinguishable from $s$ at the corresponding
proof-complete histories.

The context $\mathcal C_O(s,\pi)$ is called
\emph{uniformly admissible} for the proof $\pi$
if the following conditions hold:

\begin{enumerate}

\item  Every world
$t\in\mathcal C_O(s,\pi)$, since $s\sim_O t$, it reproduces the same
internal derivation and verification process,
ending at a corresponding Beth node $\alpha_t^\pi$.

\item Since $s\sim_O t$ the theory $T$, the inference rules,
and the conclusion $\varphi$ retain the same
intended interpretation in every world
$t\in\mathcal C_O(s,\pi)$.

\item Every axiom and assumption of $T$ used in
$\pi$ is also forced at the corresponding
proof-complete node $\alpha_t^\pi$:
\[
\alpha_t^\pi\Vdash A
\qquad
\text{for every such }A\in T.
\]

\item Every inference rule used in $\pi$
preserves Beth forcing under the intended
interpretation in every admissible continuation
of $\alpha_t^\pi$.

\end{enumerate}

These conditions must be justified independently
of the desired conclusion $\varphi$.

%In particular, worlds are not excluded from $\mathcal C_O(s,\pi)$ merely because they falsify $\varphi$.

\end{definition}

The third condition expresses the required
preservation of the theory across indistinguishable
worlds. More precisely, for every axiom or
assumption $A$ used in the proof,
\[
s\sim_O t
\quad\Longrightarrow\quad
\bigl(
\alpha_s^\pi\Vdash A
\Longleftrightarrow
\alpha_t^\pi\Vdash A
\bigr)
\]
within a uniformly admissible proof context.

%Thus, the soundness assessment does not permit the same proof to rely on an axiom that is true in one admissible world but false in another.

%This is an explicit semantic requirement of the framework, stronger than equality of accessible observation histories alone.

For each $t\in\mathcal C_O(s,\pi)$, define the
future cone of the proof-complete node by
\[
Q_t^\pi
=
\{J\in Q_t:\alpha_t^\pi\preceq_t J\}.
\]

The epistemic modality $K_O^\pi$ will be evaluated
over these proof-complete histories and their
admissible continuations. Accordingly,
\[
\alpha_s^\pi\Vdash K_O^\pi\varphi
\]
means that
\[
\forall t\in\mathcal C_O(s,\pi),
\quad
\forall J\in Q_t^\pi,
\qquad
J\Vdash\varphi.
\]

This is the Beth--Kripke knowledge clause applied
to the worlds pointed at their corresponding
post-proof histories.

\begin{theorem}[Embedded epistemic soundness]
\label{thm:embedded-beth-soundness}

Suppose that an embedded observer $O$ faithfully
constructs and verifies a finite derivation $\pi$
of $\varphi$ from an internal sound theory $T$ in a computational
world $s$.

If the associated proof context
$\mathcal C_O(s,\pi)$ is uniformly admissible,
then
\[
\boxed{
\alpha_s^\pi\Vdash K_O^\pi\varphi.
}
\tag{4}
\]

Consequently, the conclusion $\varphi$ is forced
at the proof-complete node and at every admissible
future extension of that node in every
indistinguishable world.

\end{theorem}

\begin{proof}

Let $t\in\mathcal C_O(s,\pi)$ be arbitrary.

Because $s\sim_O t$, the same proof event is
internally reproduced in $t$, and the observer
completes the derivation
\[
\pi=(\varphi_1,\ldots,\varphi_k)
\]
at the corresponding node $\alpha_t^\pi$.

We prove by induction on $j$ that
\[
\alpha_t^\pi\Vdash\varphi_j
\qquad (1\leq j\leq k).
\]

If $\varphi_j$ is an axiom or assumption used
in the derivation, the assertion follows from
the uniform admissibility of the proof context.

Otherwise, $\varphi_j$ follows from earlier
lines of $\pi$ by an inference rule of $T$.
By the induction hypothesis, the premises of
that rule are forced at $\alpha_t^\pi$.
Since the inference rule preserves Beth forcing
under the intended semantics, its conclusion
$\varphi_j$ is also forced at $\alpha_t^\pi$.

Therefore,
\[
\alpha_t^\pi\Vdash\varphi_k=\varphi.
\]

By persistence of Beth forcing,
\[
\alpha_t^\pi\preceq_t J
\quad\Longrightarrow\quad
J\Vdash\varphi.
\]

Hence
\[
\forall J\in Q_t^\pi,
\qquad
J\Vdash\varphi.
\]

Since $t$ was an arbitrary member of
$\mathcal C_O(s,\pi)$, it follows that
\[
\forall t\in\mathcal C_O(s,\pi),
\quad
\forall J\in Q_t^\pi,
\qquad
J\Vdash\varphi.
\]

By the Beth--Kripke knowledge clause,
\[
\alpha_s^\pi\Vdash K_O^\pi\varphi.
\]

\end{proof}

This theorem formalizes the principle that when an
embedded observer completes a sound proof of a
proposition, the observer knows that proposition
under the specified epistemic semantics.

The conclusion is not restricted to the actual
computational world. It holds in every world that
is indistinguishable at the proof-complete history
and satisfies the uniform semantic conditions.

Moreover, the proof event itself may evolve the
shared global state and advance the Beth node.
Once the conclusion is established, Beth persistence
guarantees that it remains forced under every
admissible future continuation.

\begin{proposition}[Bars associated with sound embedded proofs]
\label{cor:proof-requires-bar}

Under the assumptions of
Theorem~\ref{thm:embedded-beth-soundness},
for every $t\in\mathcal C_O(s,\pi)$ there
exists a bar $B_t$ above $\alpha_t^\pi$
such that
\[
\forall J\in B_t,
\qquad
J\Vdash\varphi.
\]

%In particular, such a bar exists in the actual computational world $s$.

\end{proposition}

\begin{proof}

By the Embedded Epistemic Soundness Theorem,
\[
\alpha_t^\pi\Vdash\varphi
\]
for every $t\in\mathcal C_O(s,\pi)$.

By persistence of Beth forcing, $\varphi$
is forced at every extension of $\alpha_t^\pi$.

Consequently, one may choose a bar
$B_t$ above $\alpha_t^\pi$ on which
$\varphi$ is forced at every node.
For example, under the usual definition
of a bar allowing its starting node,
$B_t=\{\alpha_t^\pi\}$ suffices.

\end{proof}

\begin{remark}[Proof events and the role of Beth bars]
\label{rem:role-of-bars}

In our framework, constructing and verifying a mathematical
proof are themselves interaction events that may evolve the
shared global state. Thus, when an embedded observer
constructs a proof $\pi$ of a formula $\varphi$, the
current Beth node may advance from $\alpha_s$ to
$\alpha_s^\pi$, while the computational world $s$
remains unchanged.

If $\pi$ is a sound proof of $\varphi$, then the
resulting node must force $\varphi$:
\[
\alpha_s^\pi\Vdash\varphi.
\]
By the Beth bar characterization, there must therefore
exist a bar $B$ above $\alpha_s^\pi$ such that
\[
\forall J\in B,\qquad J\Vdash\varphi.
\]
In particular, for an atomic proposition, Beth forcing
requires a bar on which the corresponding atomic fact
has been determined.

Consequently, a sound proof event must lead to a
global history at which its conclusion is established
throughout all admissible future continuations.

%However, Theorem~\ref{thm:beth-bar-obstruction} shows that, for the existential proposition $p_n$, no deterministic polynomial-time procedure can guarantee, under every admissible interleaving of other observers' interactions, that its proof or computation events will evolve the global state to a Beth node at which either $p_n$ or $\neg p_n$ is forced.

\end{remark}

\begin{proposition}[Counterworld criterion]
\label{cor:embedded-counterworld}

Suppose an embedded observer completes a
faithful internal derivation $\pi$ of
$\varphi$ in a computational world $s$.

If there exists a world $t\sim_O s$
with a corresponding proof-complete node
$\alpha_t^\pi$ and an admissible extension
$J\succeq_t\alpha_t^\pi$ such that
\[
J\not\Vdash\varphi,
\]
then the proof cannot satisfy all the
conditions of uniform admissibility.

In particular, if
\[
\alpha_t^\pi\Vdash\neg\varphi,
\]
the derivation cannot constitute a uniformly
sound proof of $\varphi$ across the
indistinguishable worlds.

\end{proposition}

\begin{proof}

Suppose, toward a contradiction, that the
proof context is uniformly admissible.
Then, by
Theorem~\ref{thm:embedded-beth-soundness},
\[
\alpha_s^\pi\Vdash K_O^\pi\varphi.
\]

By the Beth--Kripke knowledge clause,
\[
J\Vdash\varphi
\]
at every admissible extension of every
proof-complete node in every indistinguishable
world.
This contradicts the assumed existence of
$J\not\Vdash\varphi$.
Therefore the proof context cannot be
uniformly admissible.

\end{proof}

The Counterworld Criterion expresses the
formal content of Axiom~\ref{ax:proof}. 
If a sound proof is established by an embedded
observer, its conclusion must remain true in
every relevant indistinguishable world and
throughout the admissible future evolution of
those worlds.

%The observer need not be aware of, identify, or even conceive of the alternative worlds. Their existence and admissibility are considered in the metalanguage.

\subsection{Application and the remaining semantic bridge}

The preceding results provide a formal
Beth--Kripke interpretation of embedded
mathematical reasoning.
Within a computational world, proof construction
is represented by an evolution of its Beth node.
Across computational worlds, indistinguishability
is represented by the Kripke accessibility relation.

Embedded epistemic soundness connects these
two levels: a faithfully realized derivation
from a uniformly valid theory establishes a
conclusion that is forced throughout the relevant
future continuations of every indistinguishable
world.

We now consider how this principle applies
to the computational results of Section~\ref{sec:HSC}. 
The Beth construction establishes
Theorem~\ref{thm:beth-bar-obstruction},
which formalizes the adversarial-interleaving
mechanism underlying Appendix~B. 
 In particular, a deterministic polynomial-time
procedure cannot guarantee termination at a Beth
node above which there exists either a bar
establishing the positive answer or a bar
establishing the negative answer to the relevant
existential proposition. Indeed, there is an
admissible interleaving of observers' interactions
under which the procedure terminates at a Beth
node where neither such bar exists.
This demonstrates an obstruction to
future-stable deterministic decision under
adversarial interaction.

% However, the argument must distinguish
% the classical computational statement
% \[
% \varphi_{\mathrm{cl}}:=(P=NP)
% \]
% from the history-sensitive statement
% \[
% \varphi_{\mathrm{ev}}:=
% (P_{\mathrm{ev}}^{\mathcal A}
% =
% NP_{\mathrm{ev}}^{\mathcal A}).
% \]

% These propositions concern different
% computational semantics.

%The evolutionary separation does not by itself establish the falsity of the classical statement, and Beth--Kripke indistinguishability does not eliminate the need to preserve the intended meaning of the proposition being proved.

\begin{theorem}[Embedded-observer unprovability of $P=NP$]
\label{thm:conditional-beth-unprovability}

Assume Axioms~4.1, 4.2, 4.3, 4.5, and 7.2.
Let $O$ be an embedded mathematical observer
inhabiting a computational world $s$ in a
Beth--Kripke model $\mathfrak K$.

Suppose that every realized history in which $O$
completes a purported proof $\pi$ of
\[
\varphi_{\mathrm{cl}} := (P=NP)
\]
admits an observer-compatible evolutionary
counterrealization $t\in S$ satisfying the following
conditions:

\begin{enumerate}

\item The worlds $s$ and $t$ are internally
indistinguishable at their corresponding
proof-complete histories:
\[
s\sim_O t.
\]
In particular, the same derivation $\pi$,
proof construction, and verification process
are faithfully reproduced in $t$.

\item The theory $T$, its axioms and assumptions,
its inference rules, and the proposition
$\varphi_{\mathrm{cl}}$ retain the same intended
interpretation in both worlds. The axioms
and assumptions used in $\pi$ are forced
at the corresponding proof-complete nodes,
and the inference rules preserve Beth forcing.

\item Under that same interpretation,
the evolutionary counterrealization $t$ (in Theorem\ref{thm:beth-bar-obstruction}) does not
forces  the conclusion:
\[
\alpha_t^\pi\not\Vdash
\varphi_{\mathrm{cl}}.
\]

\end{enumerate}

Then $O$ cannot possess a uniformly sound
internal proof of the classical proposition
$P=NP$.

\end{theorem}

\begin{proof}

Suppose, toward a contradiction, that $O$
possesses a uniformly sound internal proof
$\pi$ of
\[
\varphi_{\mathrm{cl}}=(P=NP)
\]
from a theory $T$.

The construction and verification of $\pi$
are interaction events within the computational
world $s$. Hence they may evolve its global
state, moving the current Beth node from
$\alpha_s$ to a proof-complete node
$\alpha_s^\pi$, where
\[
\alpha_s\preceq_s\alpha_s^\pi.
\]

By hypothesis, there exists an observer-compatible
evolutionary counterrealization $t$ such that
\[
s\sim_O t,
\]
and the same internal proof process is faithfully
reproduced at a corresponding node
$\alpha_t^\pi$.

Since $\pi$ is assumed to be uniformly sound,
Theorem~\ref{thm:embedded-beth-soundness}
implies
\[
\alpha_s^\pi
\Vdash K_O^\pi\varphi_{\mathrm{cl}}.
\]

By the Beth--Kripke knowledge clause, the
conclusion must therefore be forced at the
corresponding proof-complete node of every
indistinguishable world, including $t$.
Consequently,
\[
\alpha_t^\pi\Vdash\varphi_{\mathrm{cl}}.
\]

However, the counterrealization hypothesis gives
\[
\alpha_t^\pi\not\Vdash
\varphi_{\mathrm{cl}}.
\]

Therefore our assumption that $\pi$ is a
uniformly sound internal proof is false. 
Hence $O$ cannot possess a uniformly sound
internal proof of $P=NP$ under the stated
hypotheses.
\end{proof}

\end{document}